\documentclass{amsart}
\usepackage{setspace}
\usepackage[margin=1in]{geometry}
\usepackage{mathrsfs}
\usepackage[T1]{fontenc}
\usepackage{enumerate}
\usepackage{amsaddr}
\usepackage[bottom]{footmisc}

\newtheorem{theorem}{Theorem}[section]
\newtheorem{corollary}{Corollary}[theorem]
\newtheorem{proposition}{Proposition}[section]
\newtheorem{lemma}[theorem]{Lemma}

\newtheorem*{remark}{Remark}

\theoremstyle{definition}

\newcommand{\RN}[1]{%
	\textup{\uppercase\expandafter{\romannumeral#1}}%
}

\def\*#1{\mathbf{#1}}

\title[On the variation of curvature functionals]{On the variation of curvature functionals in space forms with application to a generalized Willmore energy\vspace{-6ex}}
\author{Anthony Gruber$^1$, Magdalena Toda$^2$, Hung Tran$^3$ \vspace{-2ex}} 
\address{$^{1,2,3}$\,Department of Mathematics and Statistics, Texas Tech University, Lubbock, TX 79409, USA}
\thanks{Correspondence to: Anthony Gruber, Department of Mathematics and Statistics, Texas Tech University, Lubbock, TX 79409, USA \\ E-mail address: anthony.gruber@ttu.edu}

\begin{document}
\raggedbottom

\maketitle

\begin{abstract}
Functionals involving surface curvature are important across a range of scientific disciplines, and their extrema are representative of physically meaningful objects such as atomic lattices and biomembranes. Inspired in particular by the relationship of the Willmore energy to lipid bilayers, we consider a general functional depending on a surface and a symmetric combination of its principal curvatures, provided the surface is immersed in a 3-D space form of constant sectional curvature. We calculate the first and second variations of this functional, extending known results and providing computationally accessible expressions given entirely in terms of the basic geometric information found in the surface fundamental forms. Further, we motivate and introduce the p-Willmore energy functional, applying the stability criteria afforded by our calculations to prove a result about the p-Willmore energy of spheres.


\end{abstract}

\section{Introduction}
The physics of elasticity has fascinated artists, mathematicians, and scientists alike throughout recorded history.  For centuries there have been studies on how matter twists and bends in space, and mathematics has proven to be a useful tool in aiding our understanding of this phenomenon. Beginning roughly with the work of Sophie Germain on the elastic theory of surfaces in \cite{germain1821} and continuing through the contemporary work of numerous authors (such as \cite{tu2004,joshi2007,capovilla2017,elliott2017,yang2018,paragoda2018} and the references therein), the mathematics of elasticity has a rich and interesting history. Recently, developments in our understanding of biological macromolecules have renewed interest in a certain class of problems involving curvature functionals\textemdash since many elastic surfaces can be realized as the minimizers of these mathematical objects.  In particular, since electron microscopy confirmed the existence of lipid bilayers in the 1950s \cite{sjostrand1958}, there have been several curvature-centric functional models proposed for describing the dynamics of macromolecules (e.g. \cite{canham1970,helfrich1973,toda2015}). One such model was proposed by Helfrich in 1973 \cite{helfrich1973} for lipid bilayers (which are thin enough to be modeled mathematically as 2-D surfaces) and has proven to be quite reliable in approximating the behavior of biomembranes.  The Helfrich model for membrane energy per unit area is given by the functional
\begin{equation}
    E_H(M) := \int_M k_c(2H+c_0)^2 + \overline{k}K\, dS,
\end{equation}
where $H$ is the membrane mean curvature, $K$ is its Gauss curvature, $\overline{k}, k_c$ are some rigidity constants, and $c_0$ is a constant known as the "spontaneous curvature".  Physically, this high dependence on curvature arises from hydrostatic pressure differences between the fluids internal and external to the membrane.

Another noteworthy curvature functional is the bending energy, which quantifies how much (on average) a surface $M$ deviates from being a round sphere.  Specifically, the bending energy functional is defined as
\begin{equation}\label{bending}
    \mathcal{B}(M) := \frac{1}{4}\int_M (\kappa_1-\kappa_2)^2\, dS = \int_M H^2 - K + k_0 \,\, dS,
\end{equation}
where $k_0$ is the sectional curvature of the ambient space. This type of energy was first considered by Sophie Germain in 1811 (see \cite{marques2014willmore}) as a model for the bending energy of a thin plate.  In particular, she
suggested that the bending energy be measured by an integral over the plate surface, taking as integrand some symmetric and even-degree polynomial in the principal curvatures.  Note that the functional (\ref{bending}) is one of the simplest models of this kind.

\begin{remark}
The bending energy also arises in the field of computer vision, where changes in surface curvature are used to simulate natural movement. On the other hand, it is known to these scientists as the surface torsion (see \cite{sequin1995}) due to how it measures the change in normal curvature of the surface.
\end{remark}

From a mathematical perspective, both the Helfrich energy and the bending energy are closely related to the conformally invariant (see \cite{white1973}) Willmore energy popularized in \cite{willmore1965}, which is defined as 
\begin{equation}
    \mathcal{W}(M) := \int_M H^2 + k_0 \,\, dS.
\end{equation} 
In fact, the Gauss-Bonnet theorem implies that these energies differ on closed surfaces by only a constant, since in this case the integral over $K$ is completely determined by the surface topology. The Willmore energy has been widely-studied (e.g. \cite{weiner1978,bryant1984,guo2007,dziuk2008,toda2015,marques2014,mondino2014,athukorallage2015}), though there are still many open questions about its behavior. Indeed, this topic has unified the work of mathematicians, physicists, and biologists in studying elastic phenomena, and has sparked what is now an active area of research.

%

Beyond the Willmore energy, there are reasons in biology and quantum mechanics (see \cite{viswanathan1995,tu2004,siegel2006,filgueiras2012,atanasov2016}) that have led researchers to consider even more complicated curvature functionals, of which not much is yet known.  In particular, since quantum mechanical spaces frequently manifest themselves as Lorentzian manifolds, it is now of genuine physical and mathematical interest to consider surface immersions into ambient spaces different from $\mathbb{E}^3$.  Therefore, it is reasonable to approach this relatively concrete area of research from a position of generality; as physical models become more and more complicated, it will be useful to have a set of general results that can be specialized to any particular case at hand.

In accordance with this idea, it is natural to consider the functional seen in \cite{tu2004},
\begin{equation}\label{genfunc}
    \mathcal{F}(M) := \int_M \mathcal{E}(H,K) \, dS,
\end{equation}
of an immersed surface $M$ whose integrand is a general\footnote{Under some mild regularity assumptions, Newton's Theorem on symmetric polynomials implies that any symmetric polynomial in the principal curvatures $\kappa_1,\kappa_2$ of $M$ can be expressed as a smooth function $\mathcal{E}(H,K)$ of the mean and Gauss curvatures.} function of the mean curvature $H$ and Gauss curvature $K$. Taking into account the above discussion, it is further interesting to allow this immersion $M\subset \mathbb{M}^3(k_0)$ to take place in a 3-D space form of constant sectional curvature $k_0$.

General curvature functionals have been studied previously in works such as \cite{tu2004} and \cite{dogan2011} for surfaces $M\subset \mathbb{E}^3$, but as of yet these results have not been extended to more general immersions.  In light of the new consideration by mathematicians and scientists being given to surface immersions in space forms, it is worthwhile to have results that are applicable also to this more general situation.  To aid in this endeavor, the main results of this work are computationally accessible expressions for the first and second variations of $\mathcal{F}$, presented entirely in terms of classical geometric quantities.  First, note the following computational framework:

Let $U\subset \mathbb{R}^2$ be open and consider a one-parameter family of compactly supported variations of a surface $M\subset \mathbb{M}^3(k_0)$.  By reparametrizing if necessary (see \cite{clelland2017}), the variations may be assumed normal to $M$, hence given by the one-parameter family of immersions $\*r: U \times \mathbb{R} \to \mathbb{M}^3(k_0)$,
\begin{equation}\label{vary}
    \*r(\*x, t) := \*r_t(\*x) = \*r_0(\*x) + t\,u(\*x)\*N,
\end{equation}
where $\*r_0$ is the original immersion, $\*N$ is a unit normal field on $M$, and $u: U \to \mathbb{R}$ is a smooth function.  Note that since $\*r_t$ is an immersion for all $t$, the vector fields $\{\*r_i\}$ form a basis for the tangent space at each point $p\in M_t$, and that the normal velocity of this family is $\delta \*r := (d/dt)\big|_{t=0} \,\*r = u\,\*N$.  With the notation above and $h:TM\times TM\to \mathbb{R}$ denoting the shape operator of $M$, the main results are as follows.
\begin{theorem}\label{firstvar}
The first variation of the curvature functional $\mathcal{F}$ is given by
\begin{equation}\label{1var}
	\begin{split}
		&\delta \int_M \mathcal{E}(H,K) \, dS\\
		&= \int_M \bigg(\frac{1}{2}\mathcal{E}_H+ 2H\mathcal{E}_K\bigg) \Delta u + \bigg((2H^2-K+2k_0)\mathcal{E}_H+2HK\mathcal{E}_K-2H\mathcal{E}\bigg)u \\
		&\phantom{\int_M h} - \mathcal{E}_K \langle h, \text{Hess}\,u \rangle \,\, dS,
	\end{split}
\end{equation}
where $\mathcal{E}_H,\mathcal{E}_K$ denote the partial derivatives of $\mathcal{E}$ with respect to $H$ resp. $K$.
\end{theorem}

\begin{theorem}\label{secondvar}
At a critical immersion of $M$, the second variation of $\mathcal{F}$ is given by
\begin{equation}\label{2var}
	\begin{split}
	&\delta^2 \int_M \mathcal{E}(H,K)\, dS = \int_M  \bigg(\frac{1}{4}\mathcal{E}_{HH}+2H\mathcal{E}_{HK}+4H^2\mathcal{E}_{KK}+\mathcal{E}_K\bigg)(\Delta{u})^2 \, dS \\
	&+ \int_M  \mathcal{E}_{KK}\langle h,\text{Hess}\,u\rangle^2\, dS- \int_M \big(\mathcal{E}_{HK}+4H\mathcal{E}_{KK}\big)\,\Delta u \langle h,\text{Hess} \, u\rangle \, dS \\
	&+ \int_M \mathcal{E}_K\bigg(u\langle\nabla K,\nabla u\rangle -3u \langle h^2,\text{Hess}\,u\rangle -2\,h^2(\nabla u,\nabla u) - |\text{Hess}\,u|^2\bigg)\, dS  \\
	&+ \int_M\bigg((2H^2-K+2k_0)\mathcal{E}_{HH} + 2H(4H^2-K+4k_0)\mathcal{E}_{HK} + 8H^2K\mathcal{E}_{KK}\\
	&\qquad - 2H\mathcal{E}_H + (3k_0-K)\mathcal{E}_K - \mathcal{E} \bigg)u\Delta u \, dS \\
	&+ \int_M\bigg((2H^2-K+2k_0)^2\mathcal{E}_{HH}+4HK(2H^2-K+2k_0)\mathcal{E}_{HK}+4H^2K^2\mathcal{E}_{KK} \\
	&\qquad- 2K(K-2k_0)\mathcal{E}_K-2HK\mathcal{E}_H+2(K-2k_0)\mathcal{E}\bigg) u^2 \, dS \\
	&+ \int_M \big(2\mathcal{E}_H+6H\mathcal{E}_K-2(2H^2-K+2k_0)\mathcal{E}_{HK}-4HK\mathcal{E}_{KK}\big) u\langle h,\text{Hess}\, u\rangle \, dS \\
	&+ \int_M\big(\mathcal{E}_H+4H\mathcal{E}_K\big)h(\nabla u,\nabla u) \, dS + \int_M \mathcal{E}_H\, u\langle\nabla H,\nabla u\rangle\, dS \\
	&- \int_M \big(2(K-k_0)\mathcal{E}_K+H\mathcal{E}_H\big)|\nabla u|^2 \, dS,
	\end{split}
	\end{equation}
where the subscripts $\mathcal{E}_{HH}, \mathcal{E}_{HK}, \mathcal{E}_{KK}$ denote the second partial derivatives of $\mathcal{E}$ in the appropriate variables.
\end{theorem}

This provides a tool not found in current literature for studying the stability properties of elastic surfaces.  In particular, expressions (\ref{1var}) and (\ref{2var}) extend results found in \cite{tu2004} regarding the  variation of a general curvature functional for surfaces in $\mathbb{E}^3$, to surfaces immersed in an arbitrary space form.  Additionally, these expressions hold for surfaces $M$ with or without boundary, and require from their user only the computation of surface fundamental forms.  
\begin{remark}
Theorem \ref{secondvar} further provides the second variation of the Willmore energy as the special case $\mathcal{E} = H^2 + k_0$, which agrees with known results in \cite{elliott2017} and \cite{weiner1978} for immersions in $\mathbb{E}^3$ and in $S^3$, respectively.
\end{remark}

An application follows the discussion of this result, motivated by differences between the total mean curvature functional and the Willmore functional. Specifically, the \textit{$p$-Willmore energy} is introduced,
\begin{equation}
    \mathcal{W}^p(M) := \int_M H^p \, dS, \qquad p\geq 1,
\end{equation}
and the properties of round spheres are studied as a function of the exponent $p$, where it is shown that the stability of the sphere as a local minimum of $\mathcal{W}^p$ is generally dependent on the value of $p$. In particular, the accessibility of expression (\ref{2var}) allows for demonstration of the following:

\begin{theorem}\label{bigsphere}
The round sphere $S^2(r)$ immersed in Euclidean space is not a stable local minimum of $\mathcal{W}^p$ under general volume-preserving deformations for each $p > 2$. More precisely, the bilinear index form is negative definite on the eigenspace of the Laplacian associated to the first eigenvalue, and it is positive definite on the orthogonal complement subspace.
\end{theorem}

\section{Preliminaries}
The following concepts and definitions are standard in the geometric literature; for more information see \cite{carmo1992,tu2017}.  Let $\*r: U\subset \mathbb{R}^2 \to \mathbb{M}^3(k_0)$ be a surface immersion with $\*r(U) = M$, so that the vectors $\{\*r_i\}$ form a basis for the tangent space $T_qM$ at each point $q=\*r(p)$.  Let $g(\cdot, \cdot) = \langle\cdot,\cdot\rangle$ denote the Riemannian metric on $M$ inherited from the ambient \textit{space form} $\mathbb{M}^3(k_0)$ of constant sectional curvature $k_0$.
For vector fields $\*X,\*Y$, denote the covariant derivative on $\mathbb{M}^3(k_0)$ (resp. $M$) by $D$ (resp.  $\nabla$).  Recall the familiar \textit{Riemann curvature tensor} of $\mathbb{M}^3(k_0)$ defined as the operator $\overline{R}(\*X,\*Y)(\cdot) =  [D_{\*X},D_{\*Y}](\cdot) - D_{[\*X,\*Y]}(\cdot)$ where $[\cdot,\cdot]$ is the standard Lie bracket on vector fields. If $\*X,\*Y$ are orthonormal and $\sigma = \text{span}\{\*X,\*Y\}$, denote the  \textit{sectional curvature} of $\sigma$ by $K(\sigma) = \langle \overline{R}(\*X,\*Y)\*Y,\*X \rangle$.  Note that if $\*X,\*Y$ are further orthogonal on a surface $M\subset \mathbb{M}^3(k_0)$, the \textit{intrinsic Gauss curvature} $K$ can then be expressed as $K= (1/2)\big\langle R(\*X,\*Y)\*Y,\*X \big\rangle$ where $R$ is the Riemann tensor on $M$ naturally inherited from $\mathbb{M}^3(k_0)$.

If $\*N$ is a smooth unit normal field on $M \subset \mathbb{M}^3(k_0)$, recall the familiar decomposition due to Gauss (see \cite{tu2017}), $D_\*X \*Y = \nabla_\*X \*Y + \RN{2}(\*X,\*Y)$, where the tensor $\RN{2}$ is the \textit{second fundamental form} on $M$.  Recall further that since $M$ is a hypersurface, we may express the second fundamental form as $\RN{2} = h\,\*N$ where $h:TM\times TM\to\mathbb{R}$ is called the \textit{shape operator} on $M$.  The eigenvalues of its matrix representation (after contraction once with the metric inverse) are the principal curvatures $\kappa_1$ and $\kappa_2$, which together define the \textit{mean curvature} $H = (1/2)(\kappa_1+\kappa_2)$ and the \textit{extrinsic Gauss curvature} $K_E=\kappa_1\kappa_2$ of the surface.

\begin{remark}
The ``extrinsic" qualifier on $K_E$ is used above in order to distinguish this quantity from the intrinsic Gauss curvature $K$, which is independent of the immersion $\*r$.  Indeed, Gauss's Theorema Egregium asserts that $K$ is expressible entirely in terms of the metric, while $K_E= \det{\RN{2}}$ is not.  Further details are found in \cite{galvez2009}.
\end{remark}
Moreover, using $\*W,\*Z$ for two other vector fields on $\mathbb{M}^3(k_0)$, one has the essential submanifold equations of Gauss-Codazzi-Mainardi-Peterson:
\begin{align}
    & \langle \overline{R}(\*X,\*Y)\*Z,\*W \rangle = \langle R(\*X,\*Y)\*Z, \*W \rangle - \langle \RN{2}(\*X,\*W), \RN{2}(\*Y,\*Z)\rangle \nonumber\\
    &\hspace{2.4cm}+ \langle \RN{2}(\*X,\*Z), \RN{2}(\*Y,\*W)\rangle, \label{Gauss} \\
    & \langle \RN{2}(\*X,\*Y),\*N \rangle = -\langle D_\*X\*N, \*Y\rangle, \label{Weingarten} \\
    & \big(\overline{R}(\*X,\*Y)\*Z\big)^\perp = (\nabla_\*X\RN{2})(\*Y,\*Z) - (\nabla_\*Y\RN{2})(\*X,\*Z). \label{Codazzi}
\end{align}

\begin{remark}
In light of the Gauss equation (\ref{Gauss}), it is immediate that for $M \subset \mathbb{M}^3(k_0)$ one has the relationship $K_E = K - k_0$ between the extrinsic and intrinsic Gauss curvatures (see \cite{galvez2009}).
\end{remark}

Finally, note that at a regular point $p$ of the surface $M$, the exponential map $\exp_p: T_pM \to M$ affords a diffeomorphism between some neighborhoods $V\subset T_pM$ containing 0 and $W \subset M$ containing $p$.  This gives rise to a distinguished coordinate system on $\mathbb{M}^3(k_0)$ around $p$ which comes from exponentiating coordinate lines in $V$, known as \textit{normal coordinates} on $M$. These coordinates will be assumed unless otherwise stated.


\section{The First Variation}
Given a physical model such as the curvature functional $\mathcal{F}(M)$, a basic question one can ask is where it is extremized.  That is, it is important to know what surface immersions are extremal for a given functional, because their image surfaces are good candidates for physically relevant objects (see \cite{helfrich1973,toda2015,elliott2017}).  To accomplish this for $\mathcal{F}$ its first variation is computed, yielding a PDE in the mean curvature $H$.  Solutions to this equation then provide the mean curvature functions corresponding to the surface immersions of interest.

\begin{lemma}\label{evols}
Let $\varepsilon>0$ and $\mathbf{r}: M\times (-\varepsilon,\varepsilon) \to \mathbb{M}^3(k_0)$ be a smooth family of compactly supported immersions of a surface $M$ evolving with normal velocity $\delta \mathbf{r} = u\, \*N$ where $\delta = (d/dt)\big|_{t=0}$ is the variational derivative operator.  There are the following evolution equations:
    \begin{align}
        &\delta g = -2u\,h,\\
        &\delta g^{-1} = 2\,u\,\hat{h} \\
        &\delta(dS) = -2Hu \, dS, \\
        &\delta (2H) = \Delta{u}+2u(2H^2-K+2k_0), \\
    	&\delta K = 2H\Delta{u} - \langle h, \text{Hess}\,u \rangle + 2HKu,
    \end{align}
where $dS$ is the volume form on $M$, $\Delta u$ is the Laplacian of $u$ with respect to the surface metric $g$, $\langle h,\text{Hess}\,u\rangle$ is the scalar product between $h$ and the Hessian of $u$, and $\hat{h}$ is the (2,0)-tensor $g^{-1} \cdot g^{-1} \cdot h$ formed by twice contracting the shape operator with the metric inverse.
\end{lemma}
\begin{proof}
See Appendix.
\end{proof}

It is now straightforward to compute an expression for the first variation of the functional $\mathcal{F}(M)$.
\begin{proposition}\label{smallvar}
Let $M \subset \mathbb{M}^3(k_0)$ be a surface immersed in a space form of constant sectional curvature $k_0$, and let $\*r(\cdot,t)$ be a 1-parameter family of immersions as in (\ref{vary}). Then, the first variation of the curvature functional (\ref{genfunc}) is given by
 \begin{equation}
	\begin{split}
		&\delta \int_M \mathcal{E}(H,K) \, dS\\
		&= \int_M \bigg(\frac{1}{2}\mathcal{E}_H+ 2H\mathcal{E}_K\bigg) \Delta u + \bigg((2H^2-K+2k_0)\mathcal{E}_H+2HK\mathcal{E}_K-2H\mathcal{E}\bigg)u \\
		&\phantom{\int_M h} - \mathcal{E}_K \langle h, \text{Hess}\,u \rangle \,\, dS,
	\end{split}
\end{equation}
where $\mathcal{E}_H,\mathcal{E}_K$ denote the partial derivatives of $\mathcal{E}$ with respect to $H$ resp. $K$.
\end{proposition}

\begin{remark}
The expression (\ref{1var}) is also a mild extension of the work in \cite{tu2004} done for a general curvature functional of a closed surface immersed in $\mathbb{E}^3$. To see this, note that if $M$ is closed and one defines the self-adjoint operator $\nabla\cdot \tilde{\nabla}u := 2H\Delta u - \langle h,\text{Hess}\,u\rangle$ as in \cite{tu2004}, integration by parts can be applied to (\ref{1var}) to write the Euler-Lagrange equation
\begin{equation}
    \bigg(\frac{1}{2}\Delta + \big(2H^2-K+2k_0\big)\bigg)\mathcal{E}_H + \big(\nabla\cdot \tilde{\nabla} + 2HK\big)\mathcal{E}_K - 2H\mathcal{E} = 0,
\end{equation}
extending the similar expression found in \cite{tu2004} to immersions in a general space form.
\end{remark} 

\begin{proof}[Proof of Proposition \ref{smallvar}]
Using Lemma \ref{evols}, it follows that
\begin{equation}
    \begin{split}
        &\delta \int_M \mathcal{E}(H,K)\, dS = \int_M \mathcal{E}_H (\delta{H})+\mathcal{E}_K(\delta{K}) \, dS + \int_M \mathcal{E}\, \delta(dS) \\
		&= \int_M \mathcal{E}_H\bigg(\frac{1}{2}\Delta u + (2H^2-K+2k_0)u\bigg) + \mathcal{E}_K\bigg(2H\Delta u - \langle h, \text{Hess}\,u \rangle + 2HKu\bigg) \\
		&\phantom{\int_M h} - 2H\mathcal{E}u \, \, dS \\
		&= \int_M \bigg(\frac{1}{2}\mathcal{E}_H+ 2H\mathcal{E}_K\bigg)\Delta u + \bigg((2H^2-K+2k_0)\mathcal{E}_H+2HK\mathcal{E}_K-2H\mathcal{E}\bigg)u \\
		&\phantom{\int_M h} -\mathcal{E}_K \langle h, \text{Hess}\,u \rangle \,\, dS,
	\end{split}
\end{equation}
establishing the claim.
\end{proof}

\begin{remark}
Evidently, formula (\ref{1var}) agrees with known expressions for the first variation of the Willmore functional $\mathcal{W}(M)$, e.g. those found in \cite{weiner1978,bryant1984,mondino2014,athukorallage2015,elliott2017}. In each case, the Euler-Lagrange equation for closed surfaces is recovered,
\begin{equation}
    \Delta H + 2H(H^2-K+k_0) = 0,
\end{equation}
which is expected based on the literature.
\end{remark}

Note that formula (\ref{1var}) is valid for surfaces with or without boundary, providing researchers the option of restricting study to a connected subset of the surface if desired.  In particular, naturally occurring lipid bilayers may have inhomogeneous protein inclusions that affect their material properties differently across the membrane (see \cite{kik2010}), so it is beneficial to have a model that can also accommodate such analysis.  Answering further questions related to the stability of such objects requires knowledge of higher-order changes in $\mathcal{F}$, so it is reasonable to further investigate the second variation.





\section{The Second Variation}
A good expression for the second variation allows for the discussion of surface stability, which is important when drawing conclusions about physical models.  Though there may be many possible surfaces that are critical for a given curvature functional, there are frequently not as many that have the physically-desirable property of being stable under local deformations.  One example of this is seen in the catenoidal soap films that span two circular, coaxial wire loops, which are known to be minimal surfaces since surface tension forces them to be locally area-minimizing (see \cite{lautrup2011}).  In this case, stability is dependent on the sign of the second variation of the area functional, and it can be shown that for any fixed loop separation distance $z$ less than some critical value $z_0$ there are two observable catenoids that can form\textemdash only one of which is stable (see \cite{durand1981,toda2017}).  To study such stability questions in general for the functional $\mathcal{F}$, it is helpful to compute its second variation.



\begin{lemma} \label{vardelf}
Let $\varepsilon>0$ and $\mathbf{r}: M\times (-\varepsilon,\varepsilon) \to \mathbb{M}^3(k_0)$ be a smooth family of compactly supported immersions of a surface $M$ evolving with normal velocity $\delta \mathbf{r} = u\, \*N$.  For any smooth $f: M \times \mathbb{R} \to \mathbb{R}$, the evolution of the surface Laplacian $\Delta f$ and the scalar product $\langle h,\text{Hess}\,f\rangle$ are given by
\begin{align}
	\delta(\Delta f) &=\Delta\dot{f} + 2u \langle h, \text{Hess}\,f \rangle + 2u \langle \nabla H, \nabla f \rangle \nonumber \\
	&+ 2 \, h (\nabla{u},\nabla{f}) - 2H \langle \nabla u, \nabla f \rangle,
\end{align}
\begin{equation}
    \begin{split}
        \delta\langle h,\text{Hess}\,f\rangle &= \langle h,\text{Hess}\,\dot{f}\rangle + \langle\text{Hess}\,u, \text{Hess}\,f\rangle + 3u\langle h^2,\text{Hess}\,f\rangle \\
        &+ uk_0 \Delta f + 2\,h^2(\nabla u,\nabla f) + \frac{1}{2}u\langle \nabla |h|^2, \nabla f\rangle - |h|^2\langle \nabla u,\nabla f\rangle,
    \end{split}
\end{equation}
    where $\dot{f}$ denotes the partial derivative of $f$ with respect to the variational parameter $t$, and $h^2 = g^{kl}h_{li}h_{kj}\, dx^i \otimes dx^j$ is the (0,2)-tensor formed by contracting $h$ with its matrix representation.
\end{lemma}
	
\begin{proof}
While this result is known to experts in the field, its proof is not found in the literature; a computation is hence recorded here for completeness.  First, note that Einstein summation over repeated indices will be assumed throughout.  Using $f_{ij,k}$ to denote $\nabla_k f_{ij}$ and assuming a normal coordinate system, the aim is to compute the variation of the Laplacian $\Delta f = g^{ij}f_{;ij} - \Gamma_{ij}^k f_k$. To that end, one has the variation of the Christoffel symbols $\Gamma_{ij}^k$,
\begin{equation} \label{csymbvar}
	\begin{split}
		\delta\Gamma_{ij}^k &= \frac{1}{2} (\delta g^{kl})(g_{jl,i} + g_{il,j} -  g_{ij,l}) + \frac{1}{2}g^{kl}\big((\delta g_{jl})_i + (\delta g_{il})_j - (\delta g_{ij})_l\big) \\
		&= uh^{kl}(g_{jl,i} + g_{il,j} -  g_{ij,l})) - u g^{kl}(h_{jl,i} + h_{il,j} - h_{ij,l}) \\
		&\quad - g^{kl}(u_i h_{jl} + u_j h_{il} - u_l h_{ij})\\
		&= - u g^{kl}(h_{jl,i} + h_{il,j} - h_{ij,l}) - g^{kl}(u_i h_{jl} + u_j h_{il} - u_l h_{ij}).
	\end{split}
\end{equation}
It follows that
\begin{equation} \label{bigvar}
	\begin{split}
		\delta(\Delta{f}) &= \delta(g^{ij}f_{;ij}) = \delta\big(g^{ij}(f_{ij} - \Gamma_{ij}^k f_{k})\big) = \delta(g^{ij}f_{,ij}) - \delta(g^{ij}\Gamma_{ij}^k f_{k}).
	\end{split}
\end{equation}
The terms of (\ref{bigvar}) will be considered separately. Relaxing the derivative convention, it is evident that
\begin{equation} \label{metvar}
	\begin{split}
		\delta(g^{ij}f_{ij}) &= (\delta g^{ij})f_{ij} + g^{ij}(\delta f_{ij}) = 2uh^{ij}f_{ij} + g^{ij}\nabla_i\nabla_j\dot f \\
		&= 2u \langle h, \text{Hess}f \rangle + \Delta\dot f.
	\end{split}
\end{equation}
Further, it follows by (\ref{csymbvar}) that
\begin{equation} \label{inprog}
	\begin{split}
		\delta(g^{ij}&\Gamma_{ij}^k f_k) = (\delta g^{ij})\Gamma_{ij}^k f_k + g^{ij}(\delta \Gamma_{ij}^k) f_k + g^{ij}\Gamma_{ij}^k (\delta f_k) = 0 + g^{ij}(\delta \Gamma_{ij}^k) f_k + 0 \\
		&= - g^{ij}(ug^{kl}( h_{jl,i} + h_{il,j} - h_{ij,l})) f_k - g^{ij}g^{kl}(u_i h_{jl} + u_j h_{il} - u_l h_{ij})f_k \\
		&= - 2ug^{ij}g^{kl} h_{jl,i}f_k + ug^{ij}g^{kl} h_{ij,l}f_k - 2g^{ij}g^{kl}u_i h_{jl}f_k + g^{ij}g^{kl}u_l h_{ij} f_k.
	\end{split}
\end{equation}
To further simplify this, recall the Codazzi-Mainardi equation (\ref{Codazzi}), expressed in coordinate form as 
\begin{equation} \label{coordCodazzi}
    \nabla_k h_{ij} - \nabla_j h_{ik} = \overline{R}_{ijkl}N^l = k_0(\delta_{ik}\delta_{jl}-\delta_{jk}\delta_{il})N^l,
\end{equation}
where $N^l$ are the components of the unit normal vector $\*N$ and the last equality is due to the ambient space having constant sectional curvature. It follows that
\begin{equation}
	\begin{split}
	    ug^{ij}g^{kl}\nabla_jh_{il}f_k &= u(\nabla^i h^j_i)f_j = u(\nabla^j h^i_i+\overline{R}^{\,ij}_{\, \,\,\,li}N^l)f_j = 2uH^jf_j - uf_j \overline{R}^{\,j}_l N^l \\
	    &= 2uH^jf_j - uk_0f_jg^j_l N^l= 2u\langle\nabla H,\nabla f\rangle - uk_0\langle \nabla f,\mathbf{N}\rangle \\
	    &= 2u \langle \nabla H,\nabla f \rangle,
	\end{split}
\end{equation}
since $\nabla f$ is tangent to the surface and hence orthogonal to $\mathbf{N}$.  In light of this, there is now the expression
\begin{equation}
		- 2ug^{ij}g^{kl}h_{jl,i}f_k + ug^{ij}g^{kl} h_{ij,l}f_k = -ug^{ij}g^{kl} h_{ij,l}f_k,
\end{equation}
and since normal coordinates are assumed, it is seen that
\begin{equation}
		2H_l = \nabla_l(g^{ij}h_{ij}) = (\nabla_l g^{ij})h_{ij} + g^{ij}(\nabla_l h_{ij}) = 0 + g^{ij}\nabla_l h_{ij}.
\end{equation}
Therefore, (\ref{inprog}) becomes
\begin{equation} \label{christvar}
	\begin{split}
		\delta(g^{ij}\Gamma_{ij}^k f_k) &= -ug^{ij}g^{kl}h_{ij,l}f_k - g^{ij}g^{kl}u_i h_{jl}f_k - g^{ij}g^{kl}u_j h_{il}f_k + g^{ij}g^{kl}u_l h_{ij} f_k \\
		&= -2u \langle \nabla H, \nabla f \rangle - 2 \, h (\nabla{u},\nabla{f}) + 2H \langle \nabla u, \nabla f \rangle,
	\end{split}
\end{equation}
and by (\ref{metvar}) and (\ref{christvar}) there is finally
\begin{equation} \label{lapvar}
	\begin{split}
		\delta(\Delta{f}) &=  \frac{d}{dt}(g^{ij}f_{ij}) - \frac{d}{dt}(g^{ij}\Gamma_{ij}^k f_k) \\
		&= 2u \langle h, \text{Hess}f \rangle + \Delta\dot f + 2u \langle \nabla H, \nabla f \rangle + 2 \, h (\nabla{u},\nabla{f}) - 2H \langle \nabla u, \nabla f \rangle,
	\end{split}
\end{equation}
completing the first calculation.

For the further computation of the variation of the scalar product $\langle h, \text{Hess}\,f \rangle$, first notice
\begin{equation}
 \nabla_l|h|^2 =\nabla_l (h^{ij}h_{ij}) = (\nabla_lh^{ij})h_{ij}+h^{ij}(\nabla_lh_{ij}) = 2h^{ij}(\nabla_lh_{ij}),
\end{equation}
since $h^{ij} = g^{ik}g^{jm}h_{km}$ and $\nabla_lg^{ij}=0$ in normal coordinates. Also, by the Codazzi-Mainardi equation (\ref{Codazzi}),
\begin{equation}
\begin{split}
h^{ij}(\nabla_ih_{jl})u^l &= h^{ij}\nabla_i(h_{lj})u^l = h^{ij}(\nabla_lh_{ij}-R_{iljk}N^k)u^l \\
&= \frac{1}{2}\langle\nabla|h|^2,\nabla u\rangle - 2Hk_0\langle \mathbf{N},\nabla u \rangle - k_0h(\mathbf{N},\nabla u) = \frac{1}{2}\langle\nabla|h|^2,\nabla u\rangle,
\end{split}
\end{equation}
since $\mathbf{N}\perp \nabla u$ and $d\mathbf{N}(\mathbf{N}) = 0$. So, it follows that
\begin{equation}
\begin{split}
\delta \langle h,\text{Hess}\,f\rangle &= \delta \big(g^{il}g^{jk}h_{kl}(f_{ij}-\Gamma^k_{ij}f_k)\big) \\
&= 4uh^{kj}h_k^i f_{ij}+g^{il}g^{jk}\big(\nabla_k\nabla_l u - uh^s_kh_{sl}+g_{kl}uk_0\big)f_{ij} \\
&+ h^{ij}\dot{f}_{ij} - h^{ij}(\delta\Gamma^k_{ij})f_k \\
&= \langle h,\text{Hess}\,\dot{f}\rangle + \langle\text{Hess}\,u, \text{Hess}\,f\rangle + 3u\langle h^2,\text{Hess}\,f\rangle \\
& + uk_0 \Delta f + 2\,h^2(\nabla u,\nabla f) + \frac{1}{2}u\langle \nabla |h|^2, \nabla f\rangle - |h|^2\langle \nabla u,\nabla f\rangle,
\end{split}
\end{equation}
completing the calculation.
\end{proof}

It is now reasonable to present the proof of Theorem \ref{secondvar}.
\begin{proof}[Proof of Theorem \ref{secondvar}]
Combining Lemmas \ref{evols} and \ref{vardelf} with Proposition \ref{firstvar}, the aim is to compute
\begin{equation} 
\begin{split}
 \delta^2\int_M &\mathcal{E}(H,K)\, dS = \delta\int_M \delta\, \mathcal{E}(H,K)\, dS \\
&= \delta\int_M \bigg(\frac{1}{2}\mathcal{E}_H+2H\mathcal{E}_K\bigg)\Delta u + \bigg((2H^2-K+2k_0)\mathcal{E}_H +2HK\mathcal{E}_K - 2H\mathcal{E}\bigg)u \\
&\phantom{\int_M h} - \mathcal{E}_K \langle h,\text{Hess}\,u\rangle\,\, dS \\
&= \int_M \delta\, \mathcal{E}(H,K)\, \delta(dS) + \int_M \delta\bigg[\bigg(\frac{1}{2}\mathcal{E}_H + 2H\mathcal{E}_K\bigg)\Delta u\bigg]\, dS \\
&+ \int_M \delta\bigg[\bigg((2H^2-K+2k_0)\mathcal{E}_H + 2HK\mathcal{E}_K-2H\mathcal{E}\bigg)u\bigg]\, dS \\
&- \int_M \delta\bigg(\mathcal{E}_K \langle h,\text{Hess}\,u\rangle\bigg)\,dS. \\
\end{split}
\end{equation}
First, there is the term
\begin{equation} \label{deltads}
\begin{split}
\int_M \delta\, &\mathcal{E}(H,K)\, \delta(dS) \\
&=\int_M \bigg(\frac{1}{2}\mathcal{E}_H + 2H\mathcal{E}_K\bigg)\Delta u\, \delta(dS) \\ 
&+ \int_M \bigg((2H^2-K+2k_0)\mathcal{E}_H+2HK\mathcal{E}_K-2H\mathcal{E}\bigg)u \, \delta(dS) \\
&-\int_M \mathcal{E}_K\langle h,\text{Hess}\,u\rangle\, \delta(dS) \\
&= \int_M \bigg( 2H\mathcal{E}_K\, u\langle h,\text{Hess}\,u\rangle - \big(H\mathcal{E}_H+4H^2\mathcal{E}_K\big)u\Delta u \\
&\phantom{\int_M h} + \big(4H^2\mathcal{E}-4H^2K\mathcal{E}_K-2H(2H^2-K+2k_0)\mathcal{E}_H\big)u^2 \bigg) \,dS.
\end{split}
\end{equation}

$\phantom{fixing the space}$ \\
$\phantom{fixing the space}$ \\
$\phantom{fixing the space}$ \\

It is then seen that
\begin{equation} \label{first}
\begin{split}
\int_M &\delta\bigg[\bigg(\frac{1}{2}\mathcal{E}_H + 2H\mathcal{E}_K\bigg)\Delta u\bigg]\, dS \\
&= \int_M \bigg(\frac{1}{2}\mathcal{E}_H+2H\mathcal{E}_K\bigg)\delta(\Delta u)\, dS + \int_M \delta\bigg(\frac{1}{2}\mathcal{E}_H+2H\mathcal{E}_K\bigg)\Delta u \, dS \\
&= \int_M \bigg(\frac{1}{2}\mathcal{E}_H+2H\mathcal{E}_K\bigg)\Delta\dot{u}\, dS \\
&+ \int_M\bigg(\frac{\mathcal{E}_{HH}}{4}+2H\mathcal{E}_{HK}+4H^2\mathcal{E}_{KK}+\mathcal{E}_K\bigg)(\Delta u)^2 \, dS \\
&+ \int_M \bigg(\frac{\mathcal{E}_{HH}}{2}\bigg(\frac{|h|^2}{2}+k_0\bigg)+H\mathcal{E}_{HK}\big(|h|^2+K+2k_0\big) \\
&\qquad + 4H^2K\mathcal{E}_{KK} + \big(|h|^2+2k_0\big)\mathcal{E}_K\bigg) u\Delta u \, dS \\
&+ \int_M \big(\mathcal{E}_H+4H\mathcal{E}_K\big)\bigg(h(\nabla u,\nabla u)+ u\langle h,\text{Hess}\,u\rangle + u\langle \nabla H,\nabla u\rangle -H|\nabla u|^2\bigg)\, dS \\
&- \int_M \bigg(\frac{\mathcal{E}_{HK}}{2}+2H\mathcal{E}_{KK}\bigg)\Delta u \langle h,\text{Hess}\,u\rangle \, dS.
\end{split}
\end{equation}

Next, there is
\begin{equation}\label{second}
\begin{split}
\int_M &\delta\bigg[\big((2H^2-K+2k_0)\mathcal{E}_H+2HK\mathcal{E}_K-2H\mathcal{E}\big)u\bigg]\, dS \\
&= \int_M \big((2H^2-K+2k_0)\mathcal{E}_H + 2HK\mathcal{E}_K-2H\mathcal{E}\big)\dot{u} \, dS\\
&+ \int_M u\bigg[\big(2H\delta(2H)-\delta K\big)\mathcal{E}_H+(2H^2-K+2k_0)\big(\mathcal{E}_{HH}\delta H +\mathcal{E}_{HK}\delta K\big) \\
&\phantom{\int_M h} + \delta(2H)K\mathcal{E}_K 
+ 2H(\delta K)\mathcal{E}_K + 2HK\big(\mathcal{E}_{KH}\delta H + \mathcal{E}_{KK}\delta K\big) - \delta(2H)\mathcal{E} \\
&\phantom{\int_M h} -2H\big(\mathcal{E}_H\delta H + \mathcal{E}_K \delta K\big)\bigg]\, dS\\
&=  \int_M \big((2H^2-K+2k_0)\mathcal{E}_H + 2HK\mathcal{E}_K-2H\mathcal{E}\big)\dot{u} \, dS \\
&+ \int_M \bigg[\mathcal{E}_H -2HK\mathcal{E}_{KK}-\mathcal{E}_{HK}\bigg(\frac{|h|^2}{2}+k_0\bigg)\bigg]u\langle h,\text{Hess}\,u\rangle\, dS\\
&+ \int_M \bigg[\frac{\mathcal{E}_{HH}}{2}\bigg(\frac{|h|^2}{2}+k_0\bigg)+H\mathcal{E}_{HK}\big(|h|^2+K+2k_0\big) \\
&\phantom{\int_M h} + 4KH^2\mathcal{E}_{KK}-H\mathcal{E}_H+K\mathcal{E}_K-\mathcal{E}\bigg]u\Delta u\, dS \\
&+ \int_M \bigg[\frac{\mathcal{E}_{HH}}{2}\bigg(\frac{|h|^4}{2}+|h|^2k_0+2k_0^2\bigg)+2HK\mathcal{E}_{HK}\big(|h|^2+2k_0\big)+4H^2K^2\mathcal{E}_{KK} \\
&\phantom{\int_M h}+ H\mathcal{E}_H\big(|h|^2-2K+2k_0\big) + K\mathcal{E}_K\big(|h|^2+2k_0\big) - \mathcal{E}\big(|h|^2+2k_0\big)\bigg]u^2\, dS.
\end{split}
\end{equation}

Further, it follows that 
\begin{equation} \label{third}
\begin{split}
\int_M &\delta\big(\mathcal{E}_K\langle h,\text{Hess}\,u\rangle\big) \, dS \\
&= \int_M \langle h,\text{Hess}\,u\rangle\big(\mathcal{E}_{KH}\delta H + \mathcal{E}_{KK}\delta K\big) \, dS + \int_M \mathcal{E}_K\delta\langle h,\text{Hess}\,u\rangle\, dS \\
&= \int_M \mathcal{E}_K \langle h,\text{Hess}\,\dot{u}\rangle \, dS + \int_M \bigg(\frac{\mathcal{E}_{KH}}{2}+2H\mathcal{E}_K\bigg)\Delta u\langle h,\text{Hess}\,u\rangle \, dS \\
&- \int_M \mathcal{E}_{KK}\big(\langle h,\text{Hess}\,u\rangle\big)^2\, dS + \int_M \mathcal{E}_K k_0\, u\Delta u\, dS \\
&+ \int_M u\langle h,\text{Hess}\,u\rangle\bigg[\mathcal{E}_{HK}\bigg(\frac{|h|^2}{2}+k_0\bigg) + 2HK\mathcal{E}_{KK}\bigg]\, dS \\
&+ \int_M \mathcal{E}_K\bigg(3u\langle h^2,\text{Hess}\,u\rangle + |\text{Hess}\,u|^2+2\,h^2(\nabla u \nabla u) \\
&\qquad + \frac{1}{2}u\langle \nabla|h|^2,\nabla u\rangle -|h|^2|\nabla u|^2\bigg)\, dS.
\end{split}
\end{equation}

Putting together (\ref{deltads}), (\ref{first}), (\ref{third}), (\ref{second}), and noting that the first variation vanishes at a critical immersion yields (\ref{2var}), hence proving the theorem. 
\end{proof}

Since the nonnegativity of $\delta^2\mathcal{F}(M)$ is equivalent to the stability of the surface $M$ under local deformations, (\ref{2var}) provides a useful tool for studying the critical immersions of curvature functionals. Specifically, recall the bilinear \textit{index form}
\begin{equation}
    I_\mathcal{F}(M)(u,u) = \delta^2\mathcal{F}(M),
\end{equation}
the sign of which determines the stability of $M$ under local deformations.
Due to its expression only in terms of rudimentary geometric quantities, it follows that (\ref{2var}) is straightforward to apply to various specific curvature functionals in use by researchers today. In particular, it is useful in studying the following generalization of the Willmore energy.

\section{Application: $p$-Willmore energy}
Consider the total mean curvature, given by
\begin{equation}
    \mathcal{H}(M) = \int_M H\, dS.
\end{equation}
It is well known that this functional possesses different geometric properties than the Willmore energy (see e.g. \cite{almgren1998,aleksandrov2009,dalphin2016,mantoulidis2017}). In particular, $\mathcal{H}$ is not conformally invariant, and as seen in \cite{dalphin2016} spheres are minimizing for $\mathcal{H}$ only among a certain subclass of closed surfaces.  In contrast, the round sphere (of any radius) is the unique global minimizer of the Willmore energy among all closed surfaces of genus 0, as was known to Willmore himself in \cite{willmore1965}.

In light of these differences between $\mathcal{H}$ and $\mathcal{W}$, it is meaningful to consider the following question: to what extent does the power of the mean curvature $H$ appearing in the integrand of a curvature functional influence its geometric behavior?  As an application of the previous variational expressions and to obtain some partial results in this direction, consider the p-Willmore energy introduced earlier,
\begin{equation}
    \mathcal{W}^p(M) = \int_M H^p \, dS, \qquad p\geq 1.
\end{equation}
It is enlightening to examine how the properties of this functional change with the value of $p$.  Notice that the case $p=2$ recovers the usual Willmore energy functional for immersed surfaces in Euclidean space, justifying the terminology.
\begin{remark}
Note that the definition of p-Willmore energy could be further extended to include the area functional as the case $p=0$.  This would incorporate the well-studied minimal surfaces as critical points of the 0-Willmore energy.
\end{remark}


The following result is a direct consequence of Theorems \ref{firstvar} and \ref{secondvar}.
\begin{corollary}\label{pwillvar}
The first variation of $\mathcal{W}^p$ is given by
\begin{equation}\label{p1var}
\begin{split}
&\delta\int_{M}H^p \, dS= \int_M \left(\frac{p}{2}H^{p-1}\Delta u + (2H^2-K+2k_0)pH^{p-1}u-2H^{p+1}u\right) \, dS,
\end{split}
\end{equation}
Moreover, the second variation of $\mathcal{W}^p$ at a critical immersion is
\begin{equation} \label{p2var}
\begin{split}
&\delta^2\int_M H^p\, dS = \int_M \frac{p(p-1)}{4}H^{p-2}(\Delta u)^2\, dS\\ 
&+ \int_M pH^{p-1}\big(h(\nabla u,\nabla u)+2u\langle h,\text{Hess}\,u\rangle +u\langle \nabla H, \nabla u\rangle -H|\nabla u|^2\big)\, dS\\
&+ \int_M \bigg((2p^2-4p-1)H^p-p(p-1)KH^{p-2}+2p(p-1)k_0H^{p-2}\bigg)u\Delta u \, dS \\
&+ \int_M \bigg(4p(p-1)H^{p+2}-2(p-1)(2p+1)KH^p+p(p-1)K^2H^{p-2}\\
&\phantom{\int_M h} +4(2p^2-2p-1)k_0H^p-4p(p-1)k_0KH^{p-2}+4p(p-1)k_0^2H^{p-2}\bigg)u^2\, dS.
\end{split}
\end{equation}
\end{corollary}

\begin{remark}
When $M$ is closed, the Euler-Lagrange equation associated to the first variation of $\mathcal{W}^p$ is
\begin{equation}\label{HpEL}
\frac{p}{2}\Delta(H^{p-1})+ p(2H^2-K+2k_0)H^{p-1}-2H^{p+1} = 0.
\end{equation}
\end{remark}


Since the round sphere $S^2(r)$ of radius $r$ is the simplest closed surface immersed in Euclidean space, it is reasonable to consider how it behaves with respect to $\mathcal{W}^p$. Further, when discussing physical applications it is natural to allow only variations that are volume-preserving. For example, as mentioned in \cite{elliott2017} biomembranes are typically semipermeable, allowing only for the diffusion of certain ions. Therefore, when a membrane exists in a solution that has equal concentrations of solute on either of its sides, any deformation the membrane undergoes will necessarily preserve its volume.  Mathematically, this is formulated through the volume functional (see \cite{barbosa2012})
\begin{equation}\label{volfunc}
    \mathcal{V}(M) = \int_{M\times [0,t]} \*r^*(dV)
\end{equation}
where $\*r^*$ denotes pullback through the immersion $\*r$ (thought of as a map on $M\times \mathbb{R}$) and $dV$ is the volume form on $N$.  If $u$ is the normal velocity of this family, volume preservation is then imposed by requiring
\begin{equation}
    \delta \mathcal{V} = \int_M u \, dS = 0.
\end{equation}
With this perspective, (\ref{p1var}) and (\ref{p2var}) can be applied to prove Theorem \ref{bigsphere}.  To that end, note the following propositions.

\begin{proposition}\label{badsphere}
The sphere $S^2(r)$ immersed in Euclidean space is not a stable local minimum of $\mathcal{W}^p$ under volume-preserving deformations for $p > 2$. That is, for each $p > 2$, there exists a deformation $u$ such that
$\int_{S^2(r)} u\, dS = 0$, but
\begin{equation}
\delta^2 \int_{S^2(r)} H^p \, dS < 0.
\end{equation}
\end{proposition}

\begin{proof}
The volume condition is imposed by considering only variations $u$ such that $\delta \mathcal{V} = \int_{S^2(r)} u \,dS = 0.$ 
In this case, it is immediate that $S^2(r)$ is a critical point of $\mathcal{W}^p$ for any $p\geq 1$, since
\begin{equation}
\delta\int_{S^2(r)} H^p\, dS = \frac{p-2}{r^{p+1}}\int_{S^2(r)} u\, dS = 0.
\end{equation}
 On the other hand, on $S^2(r)$ the second variation becomes
\begin{equation} \label{sph}
\frac{1}{r^p}\int_{S^2(r)} \bigg(\frac{p(p-1)r^2}{4}(\Delta u)^2 + (p^2-p-1)u\Delta u + \frac{(p-1)(p-2)}{r^2}u^2\bigg)\,\, dS.
\end{equation}
Moreover, using the spectrum of the Laplacian on the sphere there is $u$ such that $\Delta u + (2/r^2)u = 0$.  In this case, the above expression is
\begin{equation}
\begin{split}
&\frac{1}{r^p}\int_{S^2(r)} \bigg(\frac{p(p-1)r^2}{4r^4}(-2u)^2  +\frac{2(p^2-p-1)}{r^2}u(-2u) + \frac{(p-1)(p-2)}{r^2}u^2\bigg)\,\, dS \\
&= \frac{1}{r^{p+2}}\int_{S^2(r)}\bigg(p(p-1)-2(p^2-p-1)+(p-1)(p-2)\bigg) u^2 \, dS \\
&= \frac{1}{r^{p+2}}\int_{S^2(r)} 2u^2(2-p)\, dS < 0, \qquad p>2,
\end{split}
\end{equation}
which proves the claim. 
\end{proof}

It is interesting to note that that the case $p=2$ is also the only case where $\mathcal{W}^p$ is invariant under conformal transformations of the ambient space.  It follows from the above proposition that if the sphere is to be minimizing among some subclass of surfaces for higher $p$, there must be further restrictions placed on the allowed variations.  To continue, note the following Poincar\'e inequality from \cite{elliott2017}.

\begin{lemma}\label{EFH}
Let $\perp$ denote orthogonality with respect to the $L^2$ inner product.  For any smooth nonconstant function $u:S^2(r) \to \mathbb{R}$ such that  $u\in\{v: \Delta v = -(2/r^2)v\}^\perp$,
	\begin{equation}
	\int_{S^2(r)} u^2 \, dS\leq\frac{r^2}{6}\int_{S^2(r)} |\nabla u|^2 \, dS\leq \frac{r^4}{36} \int_{S^2(r)} (\Delta u)^2 \, dS.
	\end{equation}
\end{lemma}
\begin{proof} Since this result is integral to the following stability analysis, a proof is presented.  Recall the solutions $\lambda_k$ with multiplicities $N_k$ to the eigenvalue problem $\Delta u + \lambda u = 0$ on $S^2(r)$ (see \cite{shubin2001}):
\begin{equation}
    \lambda_k = \frac{k(k+1)}{r^2}, \qquad N_k = {k+2\choose 2}.
\end{equation}
Clearly the constant function $1$ spans the $\lambda_0$-eigenfunctions.  Further, $\lambda_1 = 2/r^2$, so it follows that for all nonconstant $u \in \{v: \Delta v = -(2/r^2)v\}^\perp$, 
\begin{equation}
    \frac{6}{r^2} = \lambda_2 \leq \frac{\int_{S^2(r)} |\nabla u|^2\, dS}{\int_{S^2(r)} u^2\, dS},
\end{equation}
proving the first inequality.

Using $\|\cdot\|_p$ to denote the usual $L^p$ norm on $S^2(r)$, one also has
\begin{equation}
\|u\|_2^2 \leq \frac{r^2}{6}\| \nabla u\|_2^2 \leq \frac{r^2}{6}\|u\Delta u \| _1 \leq \frac{r^2}{6} \|u\|_2 \, \|\Delta u\|_2,
\end{equation}
where the second inequality is due to integration by parts and the third is by Cauchy-Schwarz.  It follows from this and the work above that 
\begin{equation}
    \|\nabla u\|_2^4 \leq \|u\|_2^2 \, \|\Delta u\|_2^2 \leq \frac{r^2}{6}\|\nabla u\|_2^2 \, \|\Delta u\|_2^2,
\end{equation}
so that $\|\nabla u\|_2^2 \leq (r^2/6)\|\Delta u\|_2^2$.  Hence,
\begin{equation}
  \|u\|_2^2 \leq \frac{r^2}{6} \|\nabla u\|_2^2 \leq \frac{r^4}{36}\|\Delta u\|_2^2,
\end{equation}
as desired. 
\end{proof}

With this estimate, it is apparent that the sphere remains a local minimum of $\mathcal{W}^p$ for $p>2$ provided consideration is given only to volume-preserving variations that are not first eigenfunctions of the surface Laplacian.

\begin{proposition}\label{goodsphere}
For all $p\geq 1$, the sphere $S^2(r)$ is a local minimum of $\mathcal{W}^p$ under volume-preserving, nonconstant deformations $u$ provided $u \in \{v : \Delta v =-(2/r^2)v \}^\perp$.  Further, the index form $I_{\mathcal{W}^p}(S^2(r))$ is coercive over this space.
\end{proposition}

\begin{proof}
By Lemma \ref{EFH}, (\ref{sph}), and integration-by-parts, it follows that
\begin{equation}
\begin{split}
   I_{\mathcal{W}^p}(S^2(r))(u,u) &= \delta^2 \int_{S^2(r)}H^p\,dS \geq \int_{S^2(r)} \frac{2p^2-3p+4}{2r^2}\,u^2 \, dS \\ 
   &\geq C(p,r)\int_{S^2(r)}u^2\, dS,
\end{split}
\end{equation}
for all allowed values of $p$. 
\end{proof}

Recall that the first eigenfunctions of $\Delta$ on $S^2(r)$ are spanned by the component functions of the position vector, which is normal to the sphere at every point.  Hence, Proposition \ref{badsphere} states that there are always volume-preserving deformations of the sphere that decrease the p-Willmore energy, and Proposition \ref{goodsphere} confirms that the only deformations which accomplish this are those that act as the components of position.  The proof of Theorem \ref{bigsphere} now follows immediately from these propositions.

\begin{remark}
    Another natural class of variations to consider are those that preserve surface area.  It is remarkable that a result completely analogous to Theorem~\ref{bigsphere} can also be proven in this case.
\end{remark}

\begin{remark}
It is interesting to note that second-order phenomena are at work here.  In particular, one can see from the second variation of the volume functional (\ref{volfunc}),
\begin{equation}
    \delta^2 \mathcal{V}(M) = \int_M -2Hu^2\, dS,
\end{equation}
that $\delta^2\mathcal{V}(M)<0$ for surfaces $M$ where $H>0$ pointwise.  Indeed, $H: S^2 \times \mathbb{R} \to \mathbb{R}$ is continuous and $\{(p,t): H(p,t)>0\}$ is an open set containing the initial (compact) sphere, so there are values $t>t_0$ where $H$ remains pointwise positive during any continuous deformation.  Hence, the volume initially decreases while the surface moves away from the sphere.
Further, for nonconstant variations $u$ of $S^2(r)$ such that $u \in \{v: \Delta v = -(2/r^2)v\}$, it follows from Lemma \ref{evols} that
\begin{equation}
    \delta(2H) = \Delta u + 2(2H^2-K)u = -\frac{2}{r^2}u + \frac{2}{r^2}u = 0.
\end{equation}
Therefore, the mean curvature (and hence the value of the functional $\mathcal{W}^p)$ does not change to first order during the deformations that cause instability.
\end{remark}

\section{Appendix}
Note the following conventions:
\begin{itemize}
    \item Einstein summation is assumed throughout, so that any index repeated twice in an expression (once up and once down) will be contracted over its appropriate range.
    \item Differentiation of a function $f$ with respect to the variable $x^j$ is denoted by $f_j$.
    \item Given a $t$-parametrized variation, the variational derivative operator is denoted by $\delta = (d/dt)\big|_{t=0}$.
\end{itemize}

\begin{proof}[Proof of Lemma \ref{evols}]
First, there is the variation of the metric:
without loss of generality, assume $\langle \*{r}_i,\*{r}_j \rangle = 0$ on $M_t$.  Using $\langle \*{N}, \*{r}_j \rangle = 0$,
	\begin{equation}
		\begin{split}
		\delta g_{ij} &= \frac{d}{dt} \big\langle \*{r}_i, \*{r}_j \big\rangle\bigg|_{t=0} = 2 \big \langle (\delta\*{r})_i, \*{r}_j \big\rangle = 2 \langle u_i \*{N} + u\*{N}_i , \*{r}_j \rangle \\
		&= 2u_i \langle \*{N}, \*{r}_j \rangle + 2u \langle \*{N}_i, \*{r}_j \rangle = -2uh_{ij}.
		\end{split}
	\end{equation}
	Since $g^{il}g_{lk}=\delta_{k}^i$, it follows that
		\begin{equation}
		(\delta g^{il})g_{lk} = -g^{il}(\delta g_{lk}) = 2ug^{il}h_{lk},
		\end{equation}
	and hence
		\begin{equation}
		\delta g^{ij} = 2ug^{il}g^{jk}h_{lk} = 2uh^{ij}.
		\end{equation}
Using this, there is the variation of the area element:
	Recall the Jacobi formula
	\begin{equation} \label{Jacobi}
	\begin{split}
	\frac{d}{dt} \det A = \det A \; \mathrm{tr} \left(A^{-1} \frac{dA}{dt}\right).
	\end{split}
	\end{equation}
Letting $\mathbf{g} = g_i^j$ be the matrix representation of the metric, it follows that
\begin{equation} \label{detg}
	\frac{d}{dt}\det(\mathbf{g}) = \det(\mathbf{g}) \; \mathrm{tr} \left(\mathbf{g}^{-1} \frac{d\mathbf{g}}{dt}\right) = \det(\mathbf{g}) (-2ug^{ij}h_{ij}) = -4 Hu \det(\mathbf{g}).
\end{equation}
Using (\ref{detg}), the variation of the surface area functional $\mathcal{A}$ is seen to be
\begin{equation} \label{surfvar}
	\begin{split}
	&\delta \mathcal{A} = \delta \int_M dS = \int_U \delta \sqrt{\det(\*g)} \, dA = \int_U \frac{1}{2\sqrt{\det(\*g)}} \, \delta (\det(\*g)) \, dA \\
	&= \int_U 2 Hu \sqrt{\det(\*g)} \, dA = \int_M -2Hu \, dS.
	\end{split}
\end{equation}
Using (\ref{surfvar}) and observing the commutativity of $d$ and $\delta$ yields the variation of the area element $dS$,
\begin{equation} \label{volform}
	\delta (dS) = d(\delta \mathcal{A}) = d \int_M -2 Hu \, dS = -2 Hu \, dS.
\end{equation}
It is now necessary to compute the variation of the shape operator $h = h_{ij}\, dx^i \otimes dx^j$. Observe,
\begin{equation} \label{hij'}
	\delta(h_{ij}) = \delta\langle \mathbf{N}, \mathbf{r}_{ij} \rangle = \langle \delta\mathbf{N}, \mathbf{r}_{ij} \rangle + \langle \mathbf{N}, \delta\mathbf{r}_{ij} \rangle.
\end{equation}
It is advantageous to compute each term of (\ref{hij'}) separately. Since $\mathbf{r}_i,\mathbf{r}_j$ is a basis for $TM$ at each point, the variation of the normal field can be expressed as $\delta\mathbf{N} = c^i\mathbf{r}_i$ for some functions $c^i$, so that
\begin{equation}
    \langle \delta\mathbf{N}, \mathbf{r}_j \rangle = \langle c^i \mathbf{r}_i, \mathbf{r}_j \rangle = c^i = -\langle \mathbf{N}, \delta\mathbf{r}_j \rangle = -\langle \mathbf{N}, u_i \mathbf{N} + u\mathbf{N}_i \rangle = -u_i,
\end{equation}
where it was used again that $\langle \mathbf{N}, \mathbf{r}_j \rangle = \langle \mathbf{N}, \mathbf{N}_j \rangle = 0$.  It follows that  $\delta\mathbf{N} = -g^{ij}u_i\mathbf{r}_j$, whereby using (\ref{Gauss}) and working in normal coordinates one sees
\begin{equation}
	\big\langle \delta\mathbf{N}, \mathbf{r}_{ij} \big\rangle = -\big\langle u^l\mathbf{r}_l, (h_{ij}\mathbf{N} + \Gamma_{ij}^k \mathbf{r}_k - g_{ij}k_0\mathbf{r}) \big\rangle = g_{ij}k_0u.
\end{equation}
Further, by (\ref{Weingarten})
\begin{equation}
	\begin{split}
		\langle \mathbf{N},\delta\mathbf{r}_{ij} \rangle &= \langle \mathbf{N}, (u\mathbf{N})_{ij} \rangle = u_{ij}+ u \langle \mathbf{N}, \mathbf{N}_{ij} \rangle = u_{ij} -u\langle h_i^l \mathbf{r}_l, h_j^k \mathbf{r}_k \rangle \\
		&= u_{ij} - uh_{il}h^l_j.
	\end{split}
\end{equation}
Therefore the variation of the second fundamental form is 
\begin{equation}
	\delta(h_{ij}) = \langle \delta\mathbf{N}, \mathbf{r}_{ij} \rangle + \langle \mathbf{N}, \delta\mathbf{r}_{ij} \rangle = g_{ij}k_0u + u_{ij} - uh_{il}h_j^l,
\end{equation}
and it is now straightforward to compute $\delta(2H)$.  Indeed, it follows that 
\begin{equation}
	\begin{split} 
		\delta (2H) &= \delta(g^{ij}h_{ij}) = (\delta g^{ij})h_{ij} + g^{ij}(\delta h_{ij}) \\
		&= 2uh^{ij}h_{ij} + g^{ij}(u_{ij} - uh_{il}h_j^l + g_{ij}k_0u)\\
		&= \Delta{u}+u(4H^2-2K+4k_0).
	\end{split}
\end{equation}
Further, there is the variation of the norm of the second fundamental form,
\begin{equation}
	\begin{split}
	    \delta |h|^2 &= \delta(h^{ij}h_{ij}) = \delta(g^{ik}g^{jl}h_{kl}h_{ij}) = 2uh^{ik}g^{jl}h_{kl}h_{ij} + 2uh^{jl}g^{ik}h_{kl}h_{ij} \\
	    &+ g^{ik}g^{jl}(g_{kl}k_0u + u_{kl} - uh_{ks}h^s_l)h_{ij} + g^{ik}g^{jl}h_{kl}(g_{ij}k_0u + u_{ij} - uh_{is}h^s_j) \\
	    &= 2uh^{ik}h^j_k h_{ij} + 4Huk_0 + 2\langle h,\text{Hess}\,u\rangle \\
	    &= 2(8H^3-6HK+8Hk_0)u + 2\langle h,\text{Hess}\,u \rangle,
	\end{split}
\end{equation}
where it was used that $8H^3 = (\kappa_1+\kappa_2)^3 = \kappa_1^3 + \kappa_2^3 + 6H(K-k_0)$.  The variation of the extrinsic Gauss curvature $K_E$ now follows, since $\delta |h|^2 = \delta(4H^2 - 2K + 2k_0)$, so
\begin{equation}
\delta K_E = 2H\,\delta(2H) - \frac{1}{2}\delta|h|^2 = 2H\Delta u - \langle h,\text{Hess}\,u\rangle + 2HKu.
\end{equation}
Since the variation of the intrinsic Gauss curvature $K$ satisfies $\delta K = \delta (K_E + k_0)$ and $k_0$ is constant, we have $\delta K = \delta K_E$, completing the calculation.
\end{proof}


\end{document}